\documentclass[11pt,reqno]{amsart}
\usepackage[T1]{fontenc}
\usepackage{graphicx}

\usepackage{color}
\definecolor{MyLinkColor}{rgb}{0,0,0.4}

\newcommand{\R}{{\mathbb R}}
\newcommand{\B}{{\mathbb B}}

\newcommand{\p}{\partial}
\newcommand{\0}{\Omega}
\newcommand{\ov}{\overline}
\newcommand{\wt}{\widetilde}
\newcommand{\wh}{\widehat}
\newcommand{\ep}{\varepsilon}

\newtheorem{thm}{Theorem}[section]

\newtheorem{lemma}[thm]{Lemma}

\theoremstyle{remark} 
\newtheorem{rem}[thm]{Remark}

\title[Steady symmetric gravity water waves]
      {A characterization of the symmetric   steady   water waves in terms of the underlying flow}

\author[Bogdan-Vasile Matioc]{}


 \email{bogdan-vasile.matioc@univie.ac.at}

\begin{document}
\maketitle

\centerline{\scshape Bogdan-Vasile Matioc}
\medskip
{\footnotesize
 \centerline{University of Vienna}
   \centerline{Nordbergstra\ss e 15}
   \centerline{1090, Vienna, Austria}
}

%

\begin{abstract}
In this paper we present a characterization of the  symmetric  rotational periodic gravity water
waves of finite depth and without stagnation points in terms of the underlying flow.
Namely, we show that such a wave  is symmetric and has a single crest and trough per period if and only if
there exists a vertical line within the fluid domain such that all
the fluid particles located on that line minimize there simultaneously their distance to the fluid bed
as they move about.
Our analysis uses the moving plane method, sharp elliptic maximum principles, and the principle of analytic continuation.
\end{abstract}

\section{Introduction}\label{S:1}

When the wind blows over a still water surface it generates waves which in time evolve into nice regular wave trains, that is, they
 become two-dimensional periodic waves that are symmetric about the crest and trough lines and   propagate at constant speed  and
without change of shape.
The study of the a priori symmetry properties of gravity water waves   goes back to  Garabedian \cite{PG65}
who used a variational approach to show that Stokes waves for which each   streamline  has a unique maximum and minimum per period
-- located beneath the crest and trough, respectively -- are symmetric.
The proof of Garabedian was simplified later on by Toland \cite{TO00}.
Assuming merely that the free surface of the irrotational flow has a single crest per period Okamoto and Sh{\=o}ji \cite{OS01} have established
the symmetry of the wave, improving thus upon the previous symmetry results.
The approach presented in \cite{OS01} is based on maximum principles and the  moving planes method, cf. \cite{BN88, JS71}, which have been applied also with
success when studying the symmetry properties of solitary waves  with and without vorticity, cf. \cite{CS88, H08, MM12}.

Concerning rotational flows, the characterization of the symmetric waves  found in \cite{OS01}  was extended
by Constantin and Escher to   finite  depth \cite{CE1} and deep water \cite{CE2} flows
without stagnation points and possessing a continuously differentiable vorticity function, under some restrictions on the
derivative of the vorticity function.
For finite depth waves without stagnation points, the latter restrictions were eliminate in \cite{CEW07}, the authors taking advantage in their 
analysis of the formulation of the problem in terms of the height function.
Sufficient conditions  which ensure the symmetry of stratified water waves can be found in \cite{Wa09}.
Concerning flows with stagnation points, it is shown in the constant vorticity case \cite{T12} that  small-amplitude gravity waves with a single crest per period are
symmetric.

Still in the context of rotational water waves of finite depth without stagnation points and  with a continuous vorticity function Hur \cite{H07} has shown that if:
$(i)$ the wave possesses a unique global  minimum per period; $(ii)$ all the streamlines attain beneath the global trough their global minimum;
and $(iii)$ the wave profile is monotone near the global trough; then the wave is symmetric and has a single crest and trough per period.
The later result was improved in \cite{MM13} where the condition $(i)$ was dropped and just one sided monotonicity was required instead of $(iii)$.

In this paper we show that in fact both conditions $(i)$ and  $(iii)$ of \cite{H07} can be omitted, and we are left with a characterization
of the  symmetry of the free surface wave  in terms of the underlying flow.
This is intriguing because the wave pattern observed at the water  surface and the underlying flow are two different aspects of the water
motion which are coupled in an extremely intricate way.
Let us emphasize that the only known explicit solution of the water wave problem is due to Gerstner \cite{Ger} (see also \cite{C4, H4}):
it is a symmetric deep water wave with only one crest per period and with all the particles located beneath  a trough
attaining  there their maximal distance to a fixed reference plane above the fluid about.
This suggests that our result could be extended to deep-water waves.
The characterization that we found is no longer valid if we allow for stagnation points: there are symmetric waves which present a Kelvin cat's eye pattern and for which
there is no vertical line within the fluid domain with the property that all
the fluid particles located on that line minimize there   their distance to the fluid bed, cf. \cite{CV11, W09}.

The outline of the paper is as follows: in Section \ref{S:2} we present the physical setting that we consider and  three equivalent mathematical formulations.
At the end of the section we present our main result Theorem \ref{T:MT}, which is proven in Section \ref{S:3}.

\section{The main result}\label{S:2}
We consider a Cartesian coordinate  system with the  $x-$axis being the direction of wave propagation, the $y$-axis pointing vertically upwards, and
we assume that the water
flow is independent of the $z$-coordinate.
From a reference frame which moves in the same direction with the wave and with the  wave speed $c$, the flow beneath the steady wave is described by
the steady Euler equations
\begin{subequations}\label{eq:P}
  \begin{equation}\label{eq:Euler}
\left\{
\begin{array}{rllll}
({u}-c) { u}_x+{ v}{ u}_y&=&-{ P}_x,\\
({ u}-c) { v}_x+{ v}{ v}_y&=&-{ P}_y-g,\\
{ u}_x+{v}_y&=&0
\end{array}
\right.\qquad \text{in $\0_\eta.$}
\end{equation}
We have taken  the fluid to be inviscid and   its  density equal to 1.
Moreover, the free surface of the wave is assumed to be the graph $y=\eta(x) $  and that the flat fluid bed is located at $y=-d$, meaning that
 the two-dimensional fluid domain   $\0_\eta $ is
\[
\0_\eta:=\{(x,y)\,:\,\text{$ x\in\R $,\,  $-d<y<\eta(x)$}\}.
\]
Furthermore, the positive constant $d$ is the average mean depth of the flow, property which implies that $\eta$ has integral mean equal to zero.
 In order to complete the mathematical system describing the motion of gravity water waves, we impose  the standard boundary conditions
\begin{equation}\label{eq:BC}
\left\{
\begin{array}{rllll}
 P&=&{P}_0&\text{on $ y=\eta(x)$},\\
 v&=&( u-c) \eta'&\text{on $ y=\eta(x)$},\\
 v&=&0 &\text{on $ y=-d$},
\end{array}
\right.
\end{equation}
 cf. \cite{ConBook, John, Kins}.
Hereby,  $ P_0$   denotes  the constant atmospheric pressure.
We shall restrict our considerations to flows for which the horizontal
velocity of each water particle is less than
the wave speed
 \begin{equation}\label{SC}
u<c\qquad\text{in $\ov \0_\eta$.}
\end{equation}
This excludes the presence of stagnation points and enables us to use equivalent formulations of the hydrodynamical problem and recent regularity properties of such flows
\cite{CE3, EM13}.
Finally,   the  vorticity of the two-dimensional flow described by the system \eqref{eq:P} is identified with the scalar function
\begin{equation*}
\omega:= { u}_y-{ v}_x\qquad\text{in $\ov\0_\eta$.}
\end{equation*}
\end{subequations}
The solutions considered in this paper correspond to periodic water waves--meaning that $\eta, u,v,P$ are periodic in $x$ and have
the same period--and possess the following
regularity
\begin{equation}\label{Reg}
 \eta\in C^2_{\rm per}(\R), \qquad u,v,P\in C^1_{\rm per}(\ov\0_\eta).
\end{equation}

The conservation of mass equation allows us to introduce the stream function $\psi\in C^2_{\rm per}(\ov\0_\eta)$ via the relations
$\nabla \psi=(-v,u-c)$ in $\ov\0_\eta$ and $\psi=0$ on $y=\eta(x)$.
Because of \eqref{SC}, one can show, cf. e.g. \cite{CS2}, that the dependence of the vorticity   on its variables takes the form
$\omega(x,y)=\gamma(-\psi(x,y))$ for all $(x,y)\in\ov\0_\eta.$
Hereby, $\gamma\in C([p_0,0])$ is the vorticity function and $p_0<0$ is the relative mass flux.
With these observations, the formulation \eqref{eq:P} can be  recast as the following problem
\begin{subequations}\label{eq:SP}
\begin{equation}\label{eq:psi}
\left\{
\begin{array}{rllll}
\Delta \psi&=&\gamma(-\psi)&\text{in}&\0_\eta,\\
\displaystyle|\nabla\psi|^2+2g(y+d)&=&Q&\text{on} &y=\eta(x),\\
\psi&=&0&\text{on}&y=\eta(x),\\
\psi&=&-p_0&\text{on} &y=-d,
\end{array}
\right.
\end{equation}
and
 \begin{equation}\label{SC1}
\psi_y<0\qquad\text{in $\ov \0_\eta$.}
\end{equation}
 \end{subequations}
The real constant $Q $ is related to the total energy of the fluid.
The equivalence of the velocity formulation \eqref{eq:P} and of the stream function formulation \eqref{eq:SP} in the setting of classical solutions
has been established   in \cite{ConBook}.
The   formulation \eqref{eq:SP} is useful because it  reduces the number of unknowns, but it also helps us to identify some properties of the underlying flow.
Particularly, because the time is eliminated from the problem the particle paths in the moving frame, that is the solutions of the ODE system
\[
\left\{
\begin{array}{lll}
 x'&=u(x,y)-c,\\
 y'&=v(x,y),
\end{array}    
\right.
\]
are also the streamlines of the steady flow and they coincide with the level curves of $\psi$.
Moreover, in view of \eqref{SC1}, they can be parametrized as graphs of periodic $C^2$- functions and they foliate the entire fluid domain.

We emphasize that due to \eqref{SC} the semi-hodograph transformation
 $\Phi:\ov\0_\eta\to\ov\0$  given by
\begin{equation}\label{semH}
\Phi(x,y):=(q,p)(x,y):=(x,-\psi(x,y))\qquad \text{for $(x,y)\in\ov\0_\eta$},
\end{equation}
where $\0:=\R\times(p_0,0),$ is a $C^2$-diffeomorphism.
Introducing the height function   $h\in C^2_{\rm per}(\ov \0)$ via the relation
\begin{equation}\label{hodo}
h(q,p):=y+d \qquad\text{for $(q,p)\in\ov\0$},
\end{equation}
 we   obtain a second equivalent formulation of \eqref{eq:P}, cf. \cite{ConBook},
\begin{subequations}\label{eq:HP}
 \begin{equation}\label{PB}
\left\{
\begin{array}{rllll}
(1+h_q^2)h_{pp}-2h_ph_qh_{pq}+h_p^2h_{qq}-\gamma h_p^3&=&0&\text{in $\0$},\\
\displaystyle 1+h_q^2+(2gh-Q)h_p^2&=&0&\text{on $p=0$},\\
h&=&0&\text{on $ p=p_0,$}
\end{array}
\right.
\end{equation}
the condition \eqref{SC} being re-expressed as
\begin{equation}\label{PBC}
 \min_{\ov \0}h_p>0.
\end{equation}
\end{subequations}
The latter formulation enables us to compare different  solutions of \eqref{eq:P} and it also helps us to  parametrize   the
streamlines of the flow, as any streamline is the graph of one of the mappings $h(\cdot,p)-d\in C^2_{\rm per}(\R),$ $p\in[p_0,0].$
Particularly, the flat bed corresponds to the choice $p=p_0$ and the free surface of the wave to $p=0.$

The main result of this paper is the following theorem.
\begin{thm} \label{T:MT}
      Consider a periodic gravity water wave solution of \eqref{eq:P} satisfying \eqref{Reg} and having a non-flat free surface.
      Moreover, assume that the vorticity function is Lipschitz continuous in $[p_1,0] $ for some $p_1\in(p_0,0). $

   Then, the wave profile is symmetric and  has only one crest and trough per period if and only if there exists a vertical line within the fluid domain such that all
   the fluid particles located on that line minimize there their distance to the fluid bed.
\end{thm}
\begin{rem}
The Theorem \ref{T:MT} concerns waves whose profiles are  symmetric with respect to the crest and  to the trough lines, the wave profile being strictly monotone between consecutive crests and troughs.
 In fact,  the symmetry of the wave profile ensures that $u$ is symmetric and $v$ is anti-symmetric with respect to the crest and to the trough lines.
 These claims can be easily deduced from properties derived in the proof of Lemma \ref{L:1}.
 \end{rem}
 \begin{rem}
 The Lipschitz continuity of the vorticity function on $[p_1,0],$  for some $p_1\in(p_0,0) $ which can be chosen  as close to $0$ as we want,
 is needed only once in the proof  of Lemma \ref{L:2} (the case $(D)$), and it seems to us that
 it cannot be omitted.
\end{rem}

\section{Proof of the main result}\label{S:3}

Our approach uses the moving plane  method: we reflect the wave surface with respect to vertical lines close to line on which the  fluid particles attain their minimal distance to the bed and move the line of reflection
to the right until an extremal position is reached.
Then, we show  that the limiting line is the unique crest line and the wave  is symmetric with respect to it.
Besides using recent regularity results for the streamlines and the profile of gravity waves, cf. \cite{EM13}, in our   analysis we employ
sharp maximum principles for elliptic partial differential equations.
For completeness, we state the following result.

\begin{lemma}[Serrin's corner point lemma]\label{L:S}
 Let
 \[R=\{(q,p)\in\R^2\,:\, a<q<b,\, p_0<p<f(q)\}
 \] where $a<b$, $f \in C^2([a,b], (p_0,\infty))$ and $f'(a)=0$ [resp. $f'(b)=0$].
 Let further $H\in C^2(\ov R)$ satisfy $\mathcal{L}H\geq 0$ in $R$ for a uniformly elliptic operator
$\mathcal{L}=a_{ij}\p_{ij}+b_i\p_i$ with continuous coefficients in $\ov R.$
Additionally, assume that there exists a positive constant $K$ such that
\begin{equation}\label{SL}
|a_{12}(q,p)|\leq K (q-a) \qquad\text{[resp. $|a_{12}(q,p)|\leq K (b-q)$]}
\end{equation}
for all $(q,p)\in R$.
If the corner point $P=(a,f(a))$ [resp. $P=(b,f(b))$] satisfies $H(P)=0$   and if $H<0$ in $R$, then either
\[
\frac{\p w}{\p s}<0\qquad\text{or}\qquad \frac{\p^2 w}{\p s^2}<0 \qquad\text{at $P$,}
\]
where $s\in R^2$ is any direction at $P$ that enters $R$ non-tangentially.
\end{lemma}
\begin{proof} See  \cite[Lemma 2]{JS71}.
\end{proof}

In order to prove the  main result, let us first observe that the assumption
that all the fluid particles located on a vertical line minimize there their distance to the fluid bed is equivalent
to saying that all the mappings $h(\cdot,p), p\in[p_0,0],$ attain their global minimum on the same vertical line.
Whence, Theorem \ref{T:MT} is  an improvement upon \cite[Theorem 3.1]{MM13} where an additional monotonicity assumption on the wave profile close to the
global trough was made.

\begin{lemma}[The necessity] \label{L:1} Assume that   $(\eta,u,v,P)$ is a solution of \eqref{eq:P} satisfying \eqref{Reg}.
If $\eta$ is symmetric and has only one crest and trough per period, then all the particles located on the trough line minimize there their distance to the fluid bed.
\end{lemma}
\begin{proof} Let $T>0$ be the minimal period of the wave.
Without restricting the generality, we may assume that $x=0 $ is the trough line, so that the crest line is the vertical line $x=T/2.$
Because the particle paths coincide with the streamlines, it suffices to show that the map $h(\cdot,p)$ is symmetric with respect to $q=T/2$ and strictly increasing on $[0,T/2]$   for each $p\in[p_0,0].$

The assumption $\eta=\eta(T-\cdot)$ implies that $h(\cdot,0)=h(T-\cdot,0).$
Taking into account that $\wt h:=h(T-\cdot,\cdot)$   is also a solution of \eqref{eq:HP},
we find that $H=\wt h-h\in C^2_{\rm per}(\ov\0)$ satisfies $H=0$ on $\p\0$ and  solves the elliptic equation $\mathcal{L}H=0$ in $\ov\0$, whereby
\begin{equation}\label{EO}
\mathcal{L}H:=(1+\wt h_q^2)H_{pp}-2\wt h_p\wt h_qH_{pq}+\wt h_p^2H_{qq}+a_1H_q+a_2H_p
\end{equation}
and
\begin{equation}\label{coeff}
\begin{aligned}
 a_1&:=(h_q+\wt h_q)  h_{pp}-2h_p h_{pq},\\
 a_2&:=(h_p+\wt h_p)  h_{qq}-2\wt h_q  h_{pq}-\gamma(h_p^2+h_p\wt h_p+\wt h^2_p).
\end{aligned}
\end{equation}
The weak elliptic maximum principle ensures that $h =h(T-\cdot,\cdot) $ in $\ov\0,$ so that all the streamlines  are symmetric with respect to the crest line $x=T/2.$
Consequently, $h_q(0,p)=h_q(T/2,p)=0$ for every $p\in[p_0,0].$
Altogether,  we find that $h_q\geq0$ on $\p R,$ whereby $R:=(0,T/2)\times (p_0,0).$

Our final goal is to use elliptic maximum principles and to show  that $h_q>0$ in $R$.
Therefore, we infer from  \cite[Proposition 2.1]{EM13} (see also \cite{CS1}),   that the distributional derivatives
$\p_q^m h\in C^1_{\rm per}(\ov\0)$ for all integers $m\geq1$, with $w:=h_q$ being the weak solution of  the elliptic equation
\begin{equation}\label{EL}
\left(\frac{1}{h_p}\p_q w\right)_q-\left(\frac{h_q}{h_p^2}\p_p w\right)_q-\left(\frac{h_q}{h_p^2}\p_q w\right)_p+
\left(\frac{1+h^2_q}{h_p^3}\p_p w\right)_p=0\qquad \text{in $\0$.}
\end{equation}
But then, it is not difficult to see these properties together with $h\in C^2_{\rm per}(\ov\0)$ yield additional regularity  $h_q\in C^2_{\rm per}(\ov\0),$
meaning that $h_q$ is a classical solution of \eqref{EL}.
The weak and strong elliptic maximum principles, cf. e.g. \cite{GT01}, ensure then that the derivative $h_q$ is strictly positive in $R$,
which is the desired claim.
\end{proof}

 We are left to prove the sufficiency claim.

\begin{lemma}[The sufficiency]\label{L:2}
 Consider a periodic gravity water wave solution $(\eta,u,v,P)$ of \eqref{eq:P} satisfying \eqref{Reg} and having a non-flat   surface.
 Moreover, assume that the vorticity function is Lipschitz continuous in $[p_1,0] $, whereby $p_1\in(p_0,0). $

   If there exists a vertical line within the fluid domain such that all
   the fluid particles located on that line minimize there their distance to the fluid bed, then the wave
   has a single crest and trough per period and is symmetric with respect to the crest line.
\end{lemma}
\begin{proof} Without loss of generality, we may assume that $q=0$ is the point where all the mappings $h(\cdot,p), p\in[p_0,0],$ attain their global minimum.
Because $h\in C ^2_{\rm per}(\ov\0)$ is a classical solution of \eqref{eq:HP}, it can be also interpreted as a weak solution of this problem,
that is $h$   satisfies the first equation of \eqref{PB} in the following weak sense
\begin{equation}\label{PB1}
 \int_\0\frac{h_q}{h_p}\varphi_q-\left(\Gamma+\frac{1+h_q^2}{2h_p^2}\right)\varphi_p\,  d(q,p)=0\qquad\text{for all $\varphi\in C^1_0(\0)$},
\end{equation}
cf. \cite{CS1, EM13}.
Hereby, $\Gamma\in C^1([p_0,0])$ is the anti-derivative of $\gamma$ satisfying  $\Gamma(0)=0$, and $ C^1_0(\0)$ is
the space containing continuously differentiable functions with compact support in $\0.$
Whence, we find all the assumptions of \cite[Theorem 3.1]{EM13} satisfied, fact which ensures us that the functions $h(\cdot,p)$ are real-analytic for all $p\in[p_0,0].$
Since $\eta=h(\cdot,0)-d,$  the principle of analytic continuation  implies that there must exist an $\ep>0$ with $\eta'>0$ on $(0,\ep).$
Therefore, for each $\lambda\in(0,\ep/2],$ the reflection $\0_\lambda^r$ of
\[\0_\lambda:=\{(x,y)\,:\, \text{$0<x<\lambda$,\, $-d<y<\eta(x)$}\}\]
with respect to  $x=\lambda$ is a subset of $\0_\eta.$

Let $\Lambda:=\max\{\lambda>0\,:\, \0_\lambda ^r\subset \0_\eta\}.$
Then,  one of the following two cases occurs: $(1)$ $\Lambda=T/2;$ $(2)$ the reflection of free wave surface with respect to $x=\Lambda$ intersects the wave surface tangentially
(say at $(x_0,\eta(x_0))$ with $x_0\in[\Lambda, 2\Lambda]$).
We define now $H:\ov R\to\R$ by setting $H(q,p)=h(2\Lambda-q,p)-h(q,p)$ for all $(q,p)$ in the rectangle $R:=(\Lambda, 2\Lambda)\times (p_0,0),$ and prove
in a first step   that $H\equiv 0$.
Indeed, our assumption ensures that $H\leq 0$  on $q=2\Lambda.$
Moreover, with our   choice of $\Lambda$ we also have that $H\leq 0 $ on $p=0.$
The last equation of \eqref{PB} and the definition of $H$ combine now to give $H\leq0$ on $\p R.$
Moreover, as $h(2\Lambda-\cdot,\cdot)$ is also a solution of \eqref{eq:HP}, we find similarly as in the proof of Lemma \ref{L:1}
that $\mathcal{L}H=0$ in $\ov R$, whereby $\mathcal{L}$ is the uniformly elliptic operator \eqref{EO}, with the modification that the
coefficients of $\mathcal{L}$ are as in \eqref{EO} and  \eqref{coeff}, but  with $\wt h:=h(2\Lambda-\cdot,\cdot).$

Let as assume by contradiction that   $H\not\equiv0.$
There are several cases to be considered separately.

{$(A)$} \, Assume first  that the case $(1)$ occurs.
Letting $P:=(2\Lambda,0)$, the weak and strong maximum principles yield that  $H(P)=0=\max_R H$ and $H<H(P)$ in $R$.
Since $\wt h_q(2\Lambda,p)=-h_q(0,p)=0$ for any $p\in[p_0,0],$  there exists a positive constant $K$ with
\begin{align*}
 | (\wt h_q\wt h_p)(q,p) |\leq K|\wt h_q(q,p)- \wt h_q (2\Lambda,p)|\leq K(2\Lambda-q)
\end{align*}
for all  $(q,p)\in R$.
We are thus in the position of applying Lemma \ref{L:S} at the corner $P$.
On the other hand,  since $\Lambda=T/2,$ the periodicity of  $H$   implies that $H_q=H_p=H_{qq}=H_{pp}=0$ at $P$.
Moreover, since $h_q(0,p)=0$ for all $p\in[p_0,0],$ we additionally have that $h_{pq}(P)=0,$ and therefore  $H_{pq}(P)=-2h_{pq}(P)=0.$
This is in contradiction with Lemma \ref{L:S}, meaning that $H\equiv 0$ in this case.\smallskip

Let us assume in the following that $\Lambda<T/2$   and that only  $(2)$ occurs.
This means that at the point $(x_0,\eta(x_0))$ where the wave surface intersects its  reflection across $x=\Lambda$ we have
\begin{equation}\label{refl}
\begin{aligned}
 &h(x_0,0)=h(2\Lambda-x_0,0),\\
 &h_q(x_0,0)=-h_q(2\Lambda-x_0,0), \\
 &h_p(x_0,0)=h_p(2\Lambda-x_0,0),
 \end{aligned}
\end{equation}
the last relation following from the previous two and the second equation of \eqref{PB}.
We distinguish the following sub-cases. \smallskip

{$(B)$} \, Assume now that we are in the case  $(2)$ and that $x_0=2\Lambda$.
Again, we choose $P:=(2\Lambda,0).$
Then, since $H\not\equiv 0,$  we obtain that    $H(P)=0 $ and $H<H(P)$ in $R$.
In order to use Serrin's corner point lemma, the same arguments as in the previous case show
\begin{align*}
 | \wt h_q\wt h_p (q,p)|\leq K(2\Lambda-q)
\end{align*}
for all $(q,p)\in R$, with $K$ being a positive constant.
Clearly, $H_q(P)=H_p(P)=0,$ so that Serrin's lemma ensures us that
\[
\frac{\p^2 H}{\p s^2}(P)<0
 \]
 for all directions $s$ at $P$ that enter $R$ non-tangentially.
Therefore, for the choice $s:=( -\wt h_p(P),-1),$ it must hold
\begin{equation}\label{PME}
\frac{\p^2 H}{\p s^2}(P)=H_{pp}+2\wt h_p H_{pq}+ \wt h_p^2 H_{qq}<0 \qquad\text{at $P$}.
\end{equation}
On the other hand, if we  differentiate the second relation of \eqref{PB} with respect to $q$, we find that
 \begin{equation}\label{REL}
 h_qh_{qq}+(2gh-Q)h_ph_{pq}+ gh_qh_p^2=0\qquad\text{on $p=0$},
\end{equation}
and therefore $h_{pq}(P)=h_{pq}(0,0)=0,$ as additionally to \eqref{refl} we know also that  $h_{q}(P)=0$ in this case.
Hence, we get $H_{pq}(P)=-h_{pq}(0,0)-h_{pq}(P)=0$.
Recalling that  $\mathcal{L}H=0$  in $\ov R,$ we have
  \begin{equation*}
H_{pp}(P)+ \wt h_p^2(P) H_{qq}(P)=0,
\end{equation*}
and together with $H_{pq}(P)=0$ we obtained a  contradiction to \eqref{PME}.
Therefore,   $H\equiv 0$ also in this case. \smallskip

{$(C)$} \, Let us now consider  the case  $(2)$ when $x_0\in(\Lambda, 2\Lambda)$.
We define $P:=(x_0,0)$ and notice that   $H\not\equiv 0$ implies that $H(P)=0>H$ in $R.$
Using Hopf's lemma, it follows that $H_p(P)>0.$
However, a direct computation shows
\[
H_p(P)=h_p(2\Lambda-x_0,0)-h_p(x_0,0)=0,
\]
cf. \eqref{refl}, which is a contradiction.
Hence, we conclude that $H\equiv 0$ in $R$.\smallskip

{$(D)$} \, We are finally left with the case  $(2)$ when $x_0=\Lambda$.
In this situation we choose $P:=(\Lambda,0)$ and note that the weak and strong maximum principles and $H\not\equiv 0$ yield  $H(P)=0>H$  in $R$, cf. \eqref{refl}.
Therefore, $h(2\Lambda-q,p)<h(q,p)$ for all $(q,p)\in R.$
We cannot apply Serrin's corner point lemma  at this stage because the relation \eqref{SL} is not satisfied by the coefficient $a_{12}$ of $\mathcal{L}.$

Instead, we  define $\Psi:\ov{ \0_\Lambda ^r}\to \R$ by setting
\[
\Psi(x,y)=\psi(2\Lambda-x,y)-\psi(x,y)\qquad\text{for $(x,y)\in\ov{ \0_\Lambda ^r}$}.
\]
Setting $\wt P:= (\Lambda, \eta(\Lambda)),$ it is obvious that $\Psi(\wt P)=0.$
We now claim that $\Psi<0$ in $\0_\Lambda ^r.$
Indeed, pick $(x,y)\in \0_\Lambda ^r$ and let $p_1, p_2\in(p_0,0)$ satisfy $p_1:=-\psi(2\Lambda-x,y)$ and $p_2:=-\psi(x,y),$ or equivalently
$y=h (2\Lambda-x,p_1)-d=h(x,p_2)-d,$ cf. \eqref{semH} and \eqref{hodo}.
But, since $(x,p_1)\in R,$ $h(x,p_1)>h(2\Lambda-x,p_1)=h(x,p_2)$, and in view of \eqref{PBC} we find that $p_1>p_2.$
Consequently, $\Psi(x,y)=p_2-p_1<0,$ which is exactly what we claimed.
We use now the Lipschitz continuity of $\gamma$ on  $[p_1,0]$, with $p_1\in(p_0,0)$,  to write
\begin{align*}
 \Delta\Psi(x,y)&=\Delta \psi(2\Lambda-x,y)-\Delta \psi(x,y)=\gamma(-\psi(2\Lambda-x,y))-\gamma(-\psi(x,y))\\
 &=c(x,y)\Psi(x,y)
\end{align*}
for all $(x,y)$ in   $\ov{\B\cap\0_\Lambda ^r},$ whereby $\B$ is
a small ball centered at the crest $(\Lambda,\eta(\Lambda)).$
Hereby,   $c:\overline{\B\cap\0_\Lambda ^r}\to\R$ is the function
\[
c(x,y)=
\left\{
\begin{array}{llll}
 \displaystyle\frac{\gamma(-\psi(2\Lambda-x,y))-\gamma(-\psi(x,y))}{\psi(2\Lambda-x,y)-\psi(x,y)}, & \text{if $\psi(x,y)\neq \psi(2\Lambda-x,y))$},\\
 0, & \text{if $\psi(x,y)=\psi(2\Lambda-x,y))$,}
\end{array}
\right.
\]
and it is   bounded   if $\B$ is chosen sufficiently small.

Finally, we introduce   the function $G:  \overline{\B\cap\0_\Lambda ^r}\to\R$ by   $G(x,y):=e^{\beta x}\Psi(x,y) $, where $\beta$ is a positive constant.
We notice that  $G(\wt P)=0> G$ in $\B\cap\0_\Lambda ^r$, and
\begin{equation}\label{SF}
\Delta G-2\beta G_x=(c-\beta^2) G\geq0\qquad\text{in $\ov{\B\cap\0_\Lambda ^r}$,}
\end{equation}
if we choose $\beta^2\geq\sup_{\B\cap\0_\Lambda ^r} c.$
Because of the special form of the elliptic operator in \eqref{SF}, we are in the position of applying Lemma \ref{L:S} and
deduce that at least one of  the first or second order partial derivatives of $G$  does not vanish at $\wt P.$
To finish, let us observe from the definition of $\Psi$ that $\Psi=\Psi_y=\Psi_{xx}=\Psi_{yy}=0$ at $\wt P.$
Moreover, as $\psi(x,\eta(x))=0$ for all $x\in\R,$ we also have that $\psi_x(\wt P)=0,$ and therewith $\Psi_x(\wt P)=0.$
Additionally,   differentiating the second equation of \eqref{eq:psi} yields
\[
\psi_x\psi_{xx}+\eta'\psi_x\psi_{xy}+\psi_y\psi_{xy}+\eta'\psi_y\psi_{yy}+g\eta'=0\qquad \text{at $\wt P,$}
\]
and, since $\psi_y<0,$ we find that $\Psi_{xy}(\wt P)=-2\psi_{xy}(\wt P)=0.$
These relations imply that all first and second order partial derivatives of $G$ vanish at $\wt P,$ which is a contradiction.
Consequently, $H\equiv0$ in $R.$

We infer  now from the analysis  in $(A)$, $(B)$, $(C)$, and $(D)$ that $H\equiv 0$ in $R$, that is $h(2\Lambda-q,p)=h(q,p)$ for all $(q,p)\in R,$
with $(\Lambda,\eta(\Lambda))$ being the first local maximum of the wave surface.
We note that the wave profile being real-analytic is also strictly increasing on $[0,\Lambda].$
Let us show that  $(\Lambda,\eta(\Lambda))$ this is  the only maximum of $\eta$ within a minimal period.
Indeed, let $\wh h:\ov\0\to\R$   be the $2\Lambda-$periodic extension of the restriction    $h\big|{[0,2\Lambda]\times [p_0,0]}.$
By  assumption, $h_q(0,p)=0$ for all $p\in[p_0,0]$, so that we also have $h_{qp}(0,p)=0$ for each $p\in[p_0,0].$
These relations together with   $h(2\Lambda-q,p)=h(q,p)$ for all $(q,p)\in R $
ensure that $\wh h\in C^2_{\rm  per}(\ov\0)$ is also a solution  of \eqref{eq:HP}.
But, the same argument as in the first part of the proof shows that $\wh h(\cdot,p)$ is a real-analytic map for all $p\in[p_0,0].$
Particularly, since $\wh h=h$ on $[0,2\Lambda]\times [p_0,0],$ the principle of analytic continuation yields that the minimal period of $\eta=h(\cdot,0)-d$ is  $2\Lambda$.
Moreover, $\eta$ has a unique maximum (and minimum) per period and is symmetric with respect to the crest (and trough) line.
This completes the proof.
\end{proof}

\section*{Acknowledgments} I am grateful for the  suggestions and the comments made by the anonymous referees.

\end{document}